\documentclass[12pt]{article}

\usepackage{setspace}
\usepackage{a4wide}
\usepackage[utf8]{inputenc}

\usepackage[english]{babel}

\usepackage{comment}
\usepackage{amsmath, amsthm, amssymb, amsfonts}
\usepackage{graphicx}
\usepackage{enumitem}
\usepackage{authblk}
\usepackage{thmtools}
\usepackage{thm-restate}
\usepackage[colorlinks=true, citecolor=green]{hyperref}
\usepackage[noabbrev,capitalize]{cleveref}
\usepackage{float}
\usepackage{xspace}
\usepackage{tikz}
\usetikzlibrary{scopes}
\usepackage{caption,subcaption}
\usepackage[sorting=none,maxnames=10]{biblatex}
\declaretheorem[name=Theorem, numberwithin=section]{theorem}
\declaretheorem[name=Lemma, sibling=theorem]{lemma}
\declaretheorem[name=Proposition, sibling=theorem]{proposition}

\declaretheorem[name=Conjecture, sibling=theorem]{conjecture}
\crefname{conjecture}{Conjecture}{Conjectures}

\crefname{hypothesis}{Hypothesis}{Hypothesis}

\crefname{claim}{Claim}{Claims}

\theoremstyle{definition}

\title{Strengthening a theorem of Meyniel}
\author{Quentin Deschamps\thanks{Univ. Lyon, Université Lyon 1, LIRIS UMR CNRS 5205, F-69621, Lyon, France.} \quad
  Carl Feghali\thanks{Univ. Lyon, EnsL, UCBL, CNRS, LIP, F-69342, Lyon Cedex 07, France. } \quad
  Franti\v{s}ek Kardo\v{s}\thanks{Faculty of Mathematics, Physics, and Informatics, Comenius University, Bratislava, Slovakia.} \quad \\
  Clément Legrand-Duchesne\thanks{LaBRI, CNRS, Université de Bordeaux, Bordeaux, France.} \quad
  Théo Pierron$^*$}

\begin{document}

\maketitle
\begin{abstract}
  For an integer $k \geq 1$ and a graph $G$, let $\mathcal{K}_k(G)$ be the graph
  that has vertex set all proper $k$-colorings of $G$, and an edge between two
  vertices $\alpha$ and~$\beta$ whenever the coloring~$\beta$ can be obtained
  from $\alpha$ by a single Kempe change.  A theorem of Meyniel from 1978 states
  that $\mathcal{K}_5(G)$ is connected with diameter $O(5^{|V(G)|})$ for every
  planar graph $G$. We significantly strengthen this result, by showing that
  there is a positive constant $c$ such that $\mathcal{K}_5(G)$ has diameter
  $O(|V(G)|^c)$ for every planar graph $G$.

\end{abstract}
\section{Introduction}

Let $k$ be a positive integer, and let $G$ be a graph.  A \emph{Kempe chain} in
colors $\{a, b\}$ is a maximal connected subgraph $B$ of $G$ such that every
vertex of $B$ has color $a$ or $b$.  By swapping the colors $a$ and $b$ on $B$,
a new coloring is obtained. This operation is called a \emph{K-change}.  Let
$\mathcal{K}_k(G)$ be the graph that has vertex set all proper $k$-colorings of
$G$, and an edge between two vertices $\alpha$ and~$\beta$ whenever the
coloring~$\beta$ can be obtained from $\alpha$ by a single K-change.

A graph $G$ is $d$-degenerate if every subgraph of $G$ contains a vertex of
degree at most $d$. Las Vergnas and Meyniel \cite{meyniel3} proved the following
result.

\begin{theorem}\label{thm:meyniel}
  If $G$ is $d$-degenerate and $k > d$ is an integer, then $\mathcal{K}_k(G)$ is
  connected.
\end{theorem}

Meyniel \cite{meyniel1} strengthened this result for planar graphs by proving

\begin{theorem}\label{thm:mm}
  If $G$ is a planar graph, then $\mathcal{K}_5(G)$ is connected.
\end{theorem}

Here, the number $5$ of colors cannot be replaced by $4$
\cite{mohar1985akempic}.

The proof of Theorem \ref{thm:meyniel} implies that $\mathcal{K}_k(G)$ has
diameter $O(d^{|V(G)|})$. Similarly, the proof of Theorem \ref{thm:mm} implies
that $\mathcal{K}_5(G)$ has diameter $O(5^{|V(G)|})$.  Bonamy, Bousquet, Feghali
and Johnson \cite{bonamyx} conjectured that the former can be significantly
improved.

\begin{conjecture}\label{con:meyniel}
  If $G = (V, E)$ is $d$-degenerate and $k > d$ is an integer, then
  $\mathcal{K}_k(G)$ has diameter $O(|V|^2)$.
\end{conjecture}

For $k > d + 1$ in Conjecture \ref{con:meyniel}, a result of Bousquet and
Heinrich \cite{heinrich} regarding the reconfiguration graph for colorings of
graphs with bounded degeneracy implies the following

\begin{theorem}
  If $G = (V, E)$ is $d$-degenerate and $k > d + 1$ is an integer, then
  $\mathcal{K}_k(G)$ has diameter $O(|V|^{d+1})$.
\end{theorem}

The case $k = d+1$ in Conjecture \ref{con:meyniel} seems much more
challenging. It was only until very recently that Bonamy, Delecroix and
Legrand-Duschene \cite{bonamy2021kempe} addressed this case for some proper
subclasses of degenerate graphs such as graphs with bounded treewidth, graphs
with bounded maximum average degree and $(\Delta -1)$-degenerate graphs, where
$\Delta$ denotes the maximum degree.

The object of this paper is to break the $k > d+1$ (in fact, $k = d + 1$)
barrier for the class of planar graphs, by proving the following strengthening
of Theorem \ref{thm:mm}.

\begin{theorem}\label{theorem}
  If $G$ is a planar graph, then $\mathcal{K}_5(G)$ has diameter at most a
  polynomial in the number of vertices of $G$.
\end{theorem}

The proof of Theorem \ref{theorem} is based on a proof method introduced in
\cite{feghali} and some ideas from \cite{eiben2020toward}, but of course has
many of its own features. We also note that the proof can be adapted for a
larger number of colors.

The paper is organized as follows. In Section \ref{pre}, we analyze a proof
from~\cite{mohar1} to get an essential estimate (Proposition~\ref{prop:m1}) that
we use in Section~\ref{sec} to prove Theorem~\ref{theorem}.

\section{The case of 3-colorable planar graphs}\label{pre}

In \cite{mohar1}, Mohar has proved the following theorem.

\begin{theorem}\label{thm:m1}
  Let $G$ be a $3$-colorable planar graph. Then $\mathcal{K}_4(G)$ is connected.
\end{theorem}

In this section we show, by making a few simple observations, that the proof of
Theorem \ref{thm:m1} in fact establishes the following stronger fact that is
crucial to our main theorem.

\begin{proposition}\label{prop:m1}
  Let $G$ be a $3$-colorable planar graph on $n$ vertices. Then for every
  $4$-colorings $\alpha$ and $\beta$ of $G$, there exists a sequence of
  K-changes from $\alpha$ to $\beta$ that changes the color of each vertex at
  most $O(n^2)$ times.
\end{proposition}

In particular, note that Proposition \ref{prop:m1} implies that
$\mathcal{K}_4(G)$ has diameter $O(n^3)$ when $G$ is a 3-colorable planar graph
on $n$ vertices. The rest of this section is devoted to the proof of
Proposition~\ref{prop:m1}. We basically follow the same steps as the proof of
Theorem~\ref{thm:m1}, except that we add some complexity estimates to:
\begin{itemize}
\item a result of Fisk~\cite{fisk} that handles the case of 3-colorable planar
  triangulations, and
\item a reduction to the (preceding) case as done by Mohar \cite{mohar1}.
\end{itemize}

\subsubsection*{Analysing the result of Fisk.}

We consider the theorem of Fisk~\cite{fisk} stated below.

\begin{theorem}\label{thm:fisk}
  Let $G$ be a $3$-colorable triangulation of the plane. Then $\mathcal{K}_4(G)$
  is connected.
\end{theorem}

We show that the proof of Theorem~\ref{thm:fisk}, word for word, gives us the
following estimate.

\begin{lemma}\label{lem:fisk}
  Let $G$ be a $3$-colorable triangulation of the plane on $n$ vertices. Then
  for every $4$-colorings $\alpha$ and $\beta$ of $G$, there exists a sequence
  of K-changes from $\alpha$ to $\beta$ that changes the color of each vertex at
  most $O(n^2)$ times.
\end{lemma}

Let $f$ be a $4$-coloring of a triangulation $G$ of the plane, and let $e =xy$
be an edge of $G$. We denote by $f(e)= \{f(x),f(y)\}$ the color of $e$ under
$f$. If $xyz$ and $xyw$ are the two triangles containing an edge $xy$, we say
that $xy$ is \emph{singular under $f$} if $f(w) = f(z)$.

To prove Lemma \ref{lem:fisk}, we require the following key structural lemma
extracted from the proof of Theorem \ref{thm:fisk}. Note that, for a
triangulation $G$ of the plane, if $G$ has a $3$-coloring, then this coloring is
unique up to permutations of colors, and all edges are singular.

\begin{lemma}\label{lem:fisk1}
  Let $G$ be a $3$-colorable triangulation of the plane, and let $f$ be a
  $4$-coloring of $G$. Then every monochromatic set of non-singular edges of $G$
  contains a cycle that bounds some region of the plane.
\end{lemma}

Using this lemma, we can conclude the proof of Lemma~\ref{lem:fisk}.
\begin{proof}[Proof of Lemma \ref{lem:fisk}]
  We shall show, by exhibiting at most $n \cdot |E(G)|$ K-changes, that any
  $4$-coloring of $G$ is K-equivalent to the (unique) $3$-coloring of $G$, which
  will prove the lemma.  To do so, given a $4$-coloring $f$ of $G$, we show how
  to obtain, via at most $n$ K-changes, a new coloring $g$ that is
  $K$-equivalent to $f$ and with fewer non-singular edges. As no edge in the
  $3$-coloring of $G$ is non-singular, by iterating this argument at most
  $|E(G)|$ times the result will follow.

  Let $e$ be a non-singular edge of $G$. By Lemma \ref{lem:fisk1}, there is a
  cycle in $G$ whose edges have the same color as $e$ and bounding some region
  $D$ of the plane.  By interchanging the two colors in $\{1,2, 3, 4 \}
  \setminus f(e)$ in the interior of $D$, we obtain a new coloring $g$ with
  fewer non-singular edges than $f$ (singular edges in the interior of $D$ stay
  singular, while the edges of the cycle on the boundary of $D$ change from
  non-singular to singular).
\end{proof}

\subsubsection*{Reduction to the triangulation case.}

We first restate Proposition 4.3 from Mohar~\cite{mohar1} except that we add
some observations about the number of vertices of the resulting triangulation
and the number of K-changes involved -- these directly follow from Mohar's
proof.

\begin{proposition}\label{lem:34}
  Let $G$ be a planar graph with a facial cycle $C$ and two 4-colorings $c_1$,
  $c_2$.  Then there exists a graph $H$ formed from $G$ by adding a
  near-triangulation of size $O(|C|)$ inside $C$ and two 4-colorings $c'_1,c'_2$
  of $H$ such that $c'_1|_{V(G)}$ and $c'_2|_{V(G)}$ are obtained from $c_1,c_2$
  using at most $O(1)$ K-changes.  Moreover, if the restriction of $c_1$ to $C$
  is a $3$-coloring, then $c'_1$ is a 3-coloring of $H$ that coincides with
  $c_1$ on $V(G)$.
\end{proposition}

We may now prove Proposition~\ref{prop:m1}.
\begin{proof}[Proof of Proposition \ref{prop:m1}]
  The proof follows the same steps as Theorem~4.4 in \cite{mohar1}. We apply
  Proposition~\ref{lem:34} to each face of $G$ (instead of Proposition~4.3 from
  \cite{mohar1}). We thus made $O(n)$ K-changes, and the resulting triangulation
  $T$ has $O(n)$ vertices. We then apply Lemma~\ref{lem:fisk} (instead of
  Theorem~4.1 from \cite{mohar1}) to obtain a sequence of K-changes between the
  two colorings of $T$ that changes the color of each vertex at most $O(n^2)$
  times.
\end{proof}

\section{Main Theorem}\label{sec}

In this section we prove Theorem \ref{theorem}. Thanks to the celebrated Four Colour Theorem, it suffices to prove the following result. 

\begin{theorem}\label{thm:2}
  Let $G$ be a plane graph with $n$ vertices. For every $5$-coloring $\alpha$ of
  $G$ and every $4$-coloring $\beta$ of $G$, there is a sequence of K-changes
  from $\alpha$ to $\beta$ where each vertex is recolored polynomially many
  times.
\end{theorem}

In the remainder of this section, we prove Theorem \ref{thm:2}. Let us briefly
sketch the details of the approach. The proof proceeds by induction on the
number of vertices. Our aim is to describe a sequence of K-changes from $\alpha$
to $\beta$ such that each vertex is recolored at most $f(n)$ times, where $f$
will satisfy a recurrence relation given at the end of the section. To establish
this, we roughly adopt the following strategy:
\begin{enumerate}
\item We find a `large' independent set $I$ that is monochromatic in both
  $\alpha$ and $\beta$ and that contains vertices of degree at most 6 in $G$
  (that $I$ is `large' will enable us to show that $f$ is a polynomial
  function).
\item We introduce an operation at a vertex that we call \emph{collapsing},
  which when applied to each vertex of $I$ gives a new graph $H$ where the
  degree of each vertex of $I$ is at most $4$ in $H$ and such that $F = H - I$
  is planar. We use these to show that any sequence of K-changes in $F$ extends
  to a sequence of K-changes in $H$ and, in turn, in $G$.
\item 
  We apply induction to find a sequence of K-changes in $F$ from any
  $5$-coloring of $F$ to a $4$-coloring of $F$ avoiding the color
  $\alpha(I)$. Applying Step 2, this sequence extends to a sequence in $G$
  ending at a $5$-coloring, where color $5$ may appear only on $I$.
\item By definition, $I \subset B$ for some color class $B$ of $\beta$; so we
  can recolor each vertex of $B$ to color $5$. Finally, noting that $G - B$ is a
  3-colorable planar graph, we then apply Proposition~\ref{prop:m1} to recolor
  the remaining vertices in $G - B$ to their color in $\beta$.
\end{enumerate}

In the rest of this section, we give the details and conclude with a small
analysis of the maximum number of times a vertex changes its color.

\subsubsection*{Step 1: Constructing $I$.}

We prove that the required independent set $I$ exists.

\begin{lemma}
  \label{lem:I}
  There exists an independent set $I$ of $G$ such that:
  \begin{itemize}
  \item all the vertices of $I$ have degree at most 6,
  \item $I$ is contained in a color class of $\alpha$ and of $\beta$,
  \item $|I|\geqslant \frac{n}{140}$.
  \end{itemize}
\end{lemma}

\begin{proof}
  Let $S$ be the set of vertices of degree at most $6$ in $G$. Then $|S| > n/7$
  since otherwise
  \[
  \sum_{v \in V(G)}d(v) \geq \sum_{v \in V(G) - S}d(v) \geq 7(n - \frac{n}{7}) =
  6n, \] which contradicts Euler's formula.

  For $i \in \{1, \dots, 5\}$ and $j \in \{1, \dots, 4\}$ define the set
  \begin{gather*}
    S_{i,j} = S \cap \alpha^{-1}(i) \cap \gamma^{-1}(j).
  \end{gather*}
  Note that each $S_{i,j}$ satisfies all the criteria from the lemma, except
  maybe the last.  However, by the pigeonhole principle, there exists $i$ and
  $j$ such that $S_{i, j}$ contains at least $|S| / (5 \times 4) \geq n/140$
  vertices, which concludes the proof.
\end{proof}

From now on, we fix a set $I$ satisfying the hypotheses of Lemma~\ref{lem:I}.

\subsubsection*{Step 2: Constructing $H$ and extending recoloring sequences.}

In order to construct $H$, we want to identify vertices in $N(I)$ that are
colored alike so that vertices of $I$ end up with degree 4 and the resulting
graph with $I$ excluded is planar. We show that we can modify the coloring
$\alpha$ so that such identification becomes possible.

Let $N$ be a plane graph. For a $5$-coloring $\varphi$ of $N$ and a vertex $v$
of $K$ with $d(v) = 6$, we say that $v$ is \emph{$\varphi$-good} if, in
$\varphi$, the vertex $v$ has three neighbors colored alike or two pairs $(a,b)$
and $(c,d)$ of neighbors colored alike that are \emph{non-overlapping},
i.e. such that $N-v+ab+cd$ is planar. A sequence of K-changes is said to
\emph{avoid} color $a$ if no vertex involved in some K-change in the sequence
changes its color to $a$.

\begin{lemma}\label{lem:swaps}
  Let $N$ be a plane graph, $\alpha$ a 5-coloring of $K$, and $v\in V(N)$ such
  that $d(v) = 6$. There exists a sequence of at most three K-changes avoiding
  $\alpha(v)$ that transforms $\alpha$ to a $5$-coloring $\beta$ of $N$ such
  that $v$ is $\beta$-good.
\end{lemma}

\begin{proof}
  We can assume that $v$ is not $\alpha$-good. Then the neighbors of $v$ can be
  colored in four possible ways (up to permutation of colors). We present a
  visual proof in Figure~\ref{fig1}, where the six neighbors of $v$ are
  represented by circles from left to right in the cyclic ordering around $v$,
  the numbers represent their color, a bold circle represents an attempt to
  perform a Kempe change, a curved edge between two vertices $u$ and $w$
  represents a Kempe chain containing both $u$ and $w$, and dashed arrows
  represent the actual Kempe changes.
\end{proof}

    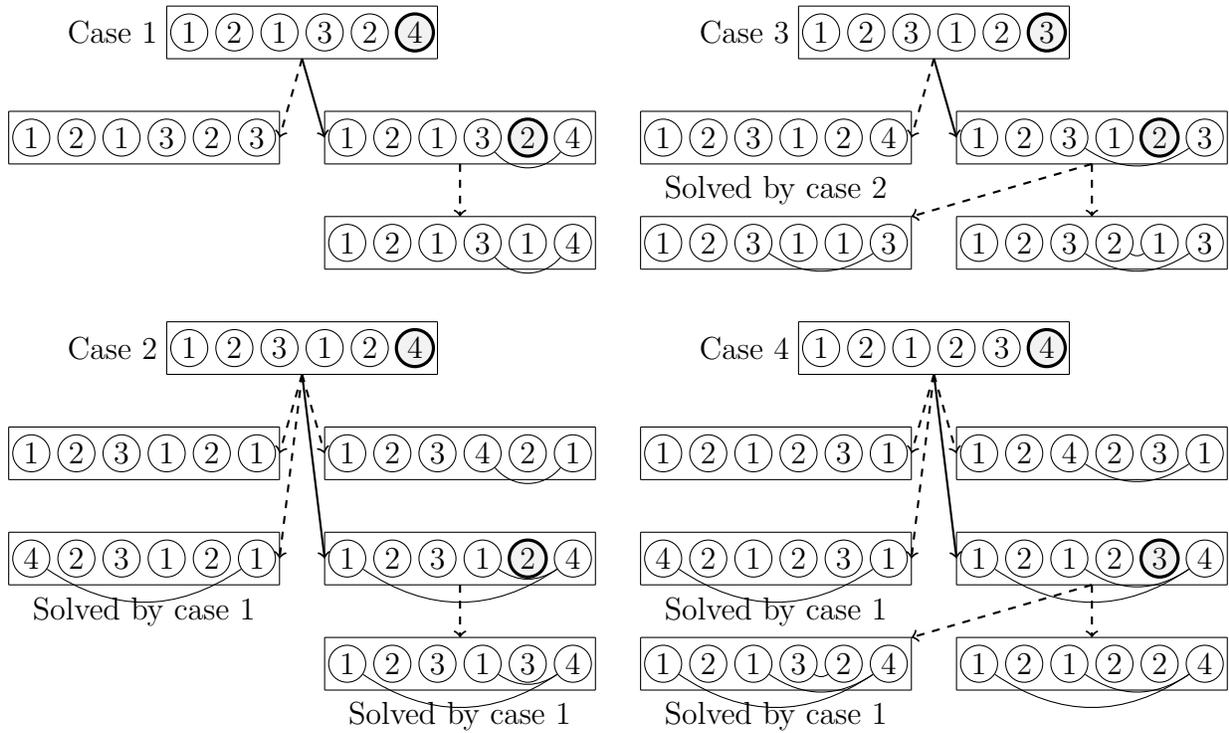
\begin{figure}[hbtp]
        \begin{center}
        \tikzstyle{vertex}=[circle,draw, minimum size=14pt, inner sep=1pt]
        \tikzstyle{change}=[circle,draw, minimum size=14pt, inner sep=1pt, fill=black!5, very thick]

        \begin{tikzpicture}[xscale=0.6,yscale=0.7]
        \node () at (-1.7,0) {Case 1};
        \node (a) at (0,0) [vertex] {1};
        \node (b) at (1,0) [vertex] {2};
        \node (c) at (2,0) [vertex] {1};
        \node (d) at (3,0) [vertex] {3};
        \node (e) at (4,0) [vertex] {2};
        \node (f) at (5,0) [change] {4};
        
        \draw[] (-0.5,0.5) to (5.5,0.5);
        \draw[] (5.5,0.5) to (5.5,-0.5);
        \draw[] (5.5,-0.5) to (-0.5,-0.5);
        \draw[] (-0.5,-0.5) to (-0.5,0.5);

        \begin{scope}[xshift=-3.5cm,yshift=-2cm]
        \node (a) at (0,0) [vertex] {1};
        \node (b) at (1,0) [vertex] {2};
        \node (c) at (2,0) [vertex] {1};
        \node (d) at (3,0) [vertex] {3};
        \node (e) at (4,0) [vertex] {2};
        \node (f) at (5,0) [vertex] {3};
        
        \draw[] (-0.5,0.5) to (5.5,0.5);
        \draw[] (5.5,0.5) to (5.5,-0.5);
        \draw[] (5.5,-0.5) to (-0.5,-0.5);
        \draw[] (-0.5,-0.5) to (-0.5,0.5);
        \end{scope}
        \begin{scope}[xshift=3.5cm,yshift=-2cm]
        \node (a) at (0,0) [vertex] {1};
        \node (b) at (1,0) [vertex] {2};
        \node (c) at (2,0) [vertex] {1};
        \node (d) at (3,0) [vertex] {3};
        \node (e) at (4,0) [change] {2};
        \node (f) at (5,0) [vertex] {4};
        
        \draw[] (-0.5,0.5) to (5.5,0.5);
        \draw[] (5.5,0.5) to (5.5,-0.5);
        \draw[] (5.5,-0.5) to (-0.5,-0.5);
        \draw[] (-0.5,-0.5) to (-0.5,0.5);
        \draw[] (d) to[bend right=45] (f);
        \end{scope}
        \begin{scope}[xshift=3.5cm,yshift=-4cm]
        \node (a) at (0,0) [vertex] {1};
        \node (b) at (1,0) [vertex] {2};
        \node (c) at (2,0) [vertex] {1};
        \node (d) at (3,0) [vertex] {3};
        \node (e) at (4,0) [vertex] {1};
        \node (f) at (5,0) [vertex] {4};
        
        \draw[] (-0.5,0.5) to (5.5,0.5);
        \draw[] (5.5,0.5) to (5.5,-0.5);
        \draw[] (5.5,-0.5) to (-0.5,-0.5);
        \draw[] (-0.5,-0.5) to (-0.5,0.5);
        \draw[] (d) to[bend right=45] (f);
        \end{scope}
        
        \draw[->,thick,dashed] (2.5,-0.5) to (2,-2);
        \draw[->,thick] (2.5,-0.5) to (3,-2);
        \draw[->,thick,dashed] (6,-2.5) to (6,-3.5);
        
        \begin{scope}[yshift=-6cm]
        \node () at (-1.7,0) {Case 2};
        \node (a) at (0,0) [vertex] {1};
        \node (b) at (1,0) [vertex] {2};
        \node (c) at (2,0) [vertex] {3};
        \node (d) at (3,0) [vertex] {1};
        \node (e) at (4,0) [vertex] {2};
        \node (f) at (5,0) [change] {4};
        
        \draw[] (-0.5,0.5) to (5.5,0.5);
        \draw[] (5.5,0.5) to (5.5,-0.5);
        \draw[] (5.5,-0.5) to (-0.5,-0.5);
        \draw[] (-0.5,-0.5) to (-0.5,0.5);
        \end{scope}        
        \begin{scope}[xshift=-3.5cm,yshift=-8cm]
        \node (a) at (0,0) [vertex] {1};
        \node (b) at (1,0) [vertex] {2};
        \node (c) at (2,0) [vertex] {3};
        \node (d) at (3,0) [vertex] {1};
        \node (e) at (4,0) [vertex] {2};
        \node (f) at (5,0) [vertex] {1};
        
        \draw[] (-0.5,0.5) to (5.5,0.5);
        \draw[] (5.5,0.5) to (5.5,-0.5);
        \draw[] (5.5,-0.5) to (-0.5,-0.5);
        \draw[] (-0.5,-0.5) to (-0.5,0.5);
        \end{scope}

        \begin{scope}[xshift=-3.5cm,yshift=-10cm]
        \node () at (2.5,-1) {Solved by case 1};
        \node (a) at (0,0) [vertex] {4};
        \node (b) at (1,0) [vertex] {2};
        \node (c) at (2,0) [vertex] {3};
        \node (d) at (3,0) [vertex] {1};
        \node (e) at (4,0) [vertex] {2};
        \node (f) at (5,0) [vertex] {1};
        
        \draw[] (-0.5,0.5) to (5.5,0.5);
        \draw[] (5.5,0.5) to (5.5,-0.5);
        \draw[] (5.5,-0.5) to (-0.5,-0.5);
        \draw[] (-0.5,-0.5) to (-0.5,0.5);
        \draw[] (a) to[bend right] (f);
        
        \end{scope}
        \begin{scope}[xshift=3.5cm,yshift=-8cm]
        \node (a) at (0,0) [vertex] {1};
        \node (b) at (1,0) [vertex] {2};
        \node (c) at (2,0) [vertex] {3};
        \node (d) at (3,0) [vertex] {4};
        \node (e) at (4,0) [vertex] {2};
        \node (f) at (5,0) [vertex] {1};
        
        \draw[] (-0.5,0.5) to (5.5,0.5);
        \draw[] (5.5,0.5) to (5.5,-0.5);
        \draw[] (5.5,-0.5) to (-0.5,-0.5);
        \draw[] (-0.5,-0.5) to (-0.5,0.5);
        \draw[] (d) to[bend right=45] (f);
        \end{scope}     
        \begin{scope}[xshift=3.5cm,yshift=-10cm]
        \node (a) at (0,0) [vertex] {1};
        \node (b) at (1,0) [vertex] {2};
        \node (c) at (2,0) [vertex] {3};
        \node (d) at (3,0) [vertex] {1};
        \node (e) at (4,0) [change] {2};
        \node (f) at (5,0) [vertex] {4};
        
        \draw[] (-0.5,0.5) to (5.5,0.5);
        \draw[] (5.5,0.5) to (5.5,-0.5);
        \draw[] (5.5,-0.5) to (-0.5,-0.5);
        \draw[] (-0.5,-0.5) to (-0.5,0.5);
        \draw[] (d) to[bend right] (f);  
        \draw[] (a) to[bend right] (f); 
        \end{scope}
        
        \begin{scope}[xshift=3.5cm,yshift=-12cm]
        \node (a) at (0,0) [vertex] {1};
        \node (b) at (1,0) [vertex] {2};
        \node (c) at (2,0) [vertex] {3};
        \node (d) at (3,0) [vertex] {1};
        \node (e) at (4,0) [vertex] {3};
        \node (f) at (5,0) [vertex] {4};
        \node () at (2.5,-1) {Solved by case 1};
        
        \draw[] (-0.5,0.5) to (5.5,0.5);
        \draw[] (5.5,0.5) to (5.5,-0.5);
        \draw[] (5.5,-0.5) to (-0.5,-0.5);
        \draw[] (-0.5,-0.5) to (-0.5,0.5);
        \draw[] (d) to[bend right] (f);  
        \draw[] (a) to[bend right] (f); 
        \end{scope}
        
        \draw[->,thick,dashed] (2.5,-6.5) to (2,-8);
        \draw[->,thick,dashed] (2.5,-6.5) to (2,-10);
        \draw[->,thick,dashed] (2.5,-6.5) to (3,-8);
        \draw[->,thick] (2.5,-6.5) to (3,-10);
        \draw[->,thick,dashed] (6,-10.5) to (6,-11.5);
        \begin{scope}[xshift=14cm,yshift=-6cm]
        \node () at (-1.7,0) {Case 4};
        \node (a) at (0,0) [vertex] {1};
        \node (b) at (1,0) [vertex] {2};
        \node (c) at (2,0) [vertex] {1};
        \node (d) at (3,0) [vertex] {2};
        \node (e) at (4,0) [vertex] {3};
        \node (f) at (5,0) [change] {4};
        
        \draw[] (-0.5,0.5) to (5.5,0.5);
        \draw[] (5.5,0.5) to (5.5,-0.5);
        \draw[] (5.5,-0.5) to (-0.5,-0.5);
        \draw[] (-0.5,-0.5) to (-0.5,0.5);
        \end{scope}
        
        \begin{scope}[xshift=10.5cm,yshift=-8cm]
        \node (a) at (0,0) [vertex] {1};
        \node (b) at (1,0) [vertex] {2};
        \node (c) at (2,0) [vertex] {1};
        \node (d) at (3,0) [vertex] {2};
        \node (e) at (4,0) [vertex] {3};
        \node (f) at (5,0) [vertex] {1};
        
        \draw[] (-0.5,0.5) to (5.5,0.5);
        \draw[] (5.5,0.5) to (5.5,-0.5);
        \draw[] (5.5,-0.5) to (-0.5,-0.5);
        \draw[] (-0.5,-0.5) to (-0.5,0.5);
        \end{scope}
        \begin{scope}[xshift=10.5cm,yshift=-10cm]
        \node (a) at (0,0) [vertex] {4};
        \node (b) at (1,0) [vertex] {2};
        \node (c) at (2,0) [vertex] {1};
        \node (d) at (3,0) [vertex] {2};
        \node (e) at (4,0) [vertex] {3};
        \node (f) at (5,0) [vertex] {1};
        \node () at (2.5,-1) {Solved by case 1};
        
        \draw[] (-0.5,0.5) to (5.5,0.5);
        \draw[] (5.5,0.5) to (5.5,-0.5);
        \draw[] (5.5,-0.5) to (-0.5,-0.5);
        \draw[] (-0.5,-0.5) to (-0.5,0.5);
        \draw[] (a) to[bend right] (f);
        \end{scope}
        \begin{scope}[xshift=17.5cm,yshift=-8cm]

        \node (a) at (0,0) [vertex] {1};
        \node (b) at (1,0) [vertex] {2};
        \node (c) at (2,0) [vertex] {4};
        \node (d) at (3,0) [vertex] {2};
        \node (e) at (4,0) [vertex] {3};
        \node (f) at (5,0) [vertex] {1};
        
        \draw[] (-0.5,0.5) to (5.5,0.5);
        \draw[] (5.5,0.5) to (5.5,-0.5);
        \draw[] (5.5,-0.5) to (-0.5,-0.5);
        \draw[] (-0.5,-0.5) to (-0.5,0.5);
        \draw[] (c) to[bend right] (f);
        \end{scope}
        \begin{scope}[xshift=17.5cm,yshift=-10cm]
        \node (a) at (0,0) [vertex] {1};
        \node (b) at (1,0) [vertex] {2};
        \node (c) at (2,0) [vertex] {1};
        \node (d) at (3,0) [vertex] {2};
        \node (e) at (4,0) [change] {3};
        \node (f) at (5,0) [vertex] {4};
        
        \draw[] (-0.5,0.5) to (5.5,0.5);
        \draw[] (5.5,0.5) to (5.5,-0.5);
        \draw[] (5.5,-0.5) to (-0.5,-0.5);
        \draw[] (-0.5,-0.5) to (-0.5,0.5);
        \draw[] (a) to[bend right] (f);
        \draw[] (c) to[bend right] (f);
        \end{scope}

        \begin{scope}[xshift=17.5cm,yshift=-12cm]
        \node (a) at (0,0) [vertex] {1};
        \node (b) at (1,0) [vertex] {2};
        \node (c) at (2,0) [vertex] {1};
        \node (d) at (3,0) [vertex] {2};
        \node (e) at (4,0) [vertex] {2};
        \node (f) at (5,0) [vertex] {4};
        
        \draw[] (-0.5,0.5) to (5.5,0.5);
        \draw[] (5.5,0.5) to (5.5,-0.5);
        \draw[] (5.5,-0.5) to (-0.5,-0.5);
        \draw[] (-0.5,-0.5) to (-0.5,0.5);
        \draw[] (a) to[bend right] (f);
        \draw[] (c) to[bend right] (f);
        \end{scope}
        \begin{scope}[xshift=10.5cm,yshift=-12cm]
        \node (a) at (0,0) [vertex] {1};
        \node (b) at (1,0) [vertex] {2};
        \node (c) at (2,0) [vertex] {1};
        \node (d) at (3,0) [vertex] {3};
        \node (e) at (4,0) [vertex] {2};
        \node (f) at (5,0) [vertex] {4};
        \node () at (2.5,-1) {Solved by case 1};

        \draw[] (-0.5,0.5) to (5.5,0.5);
        \draw[] (5.5,0.5) to (5.5,-0.5);
        \draw[] (5.5,-0.5) to (-0.5,-0.5);
        \draw[] (-0.5,-0.5) to (-0.5,0.5);
        \draw[] (a) to[bend right] (f);
        \draw[] (c) to[bend right] (f);
        \draw[] (e) to[bend left] (d);
        \end{scope}
         
        \draw[->,thick,dashed] (16.5,-6.5) to (16,-8);
        \draw[->,thick,dashed] (16.5,-6.5) to (17,-8);
        \draw[->,thick,dashed] (16.5,-6.5) to (16,-10);
        \draw[->,thick] (16.5,-6.5) to (17,-10);       
        \draw[->,thick,dashed] (20,-10.5) to (20,-11.5); 
        \draw[->,thick,dashed] (20,-10.5) to (16,-11.5); 

        \begin{scope}[xshift=14cm]
        \node () at (-1.7,0) {Case 3};
        \node (a) at (0,0) [vertex] {1};
        \node (b) at (1,0) [vertex] {2};
        \node (c) at (2,0) [vertex] {3};
        \node (d) at (3,0) [vertex] {1};
        \node (e) at (4,0) [vertex] {2};
        \node (f) at (5,0) [change] {3};
        
        \draw[] (-0.5,0.5) to (5.5,0.5);
        \draw[] (5.5,0.5) to (5.5,-0.5);
        \draw[] (5.5,-0.5) to (-0.5,-0.5);
        \draw[] (-0.5,-0.5) to (-0.5,0.5);

        \end{scope}

        \begin{scope}[xshift=10.5cm,yshift=-2cm]
        \node () at (2.5,-1) {Solved by case 2};
        \node (a) at (0,0) [vertex] {1};
        \node (b) at (1,0) [vertex] {2};
        \node (c) at (2,0) [vertex] {3};
        \node (d) at (3,0) [vertex] {1};
        \node (e) at (4,0) [vertex] {2};
        \node (f) at (5,0) [vertex] {4};
        
        \draw[] (-0.5,0.5) to (5.5,0.5);
        \draw[] (5.5,0.5) to (5.5,-0.5);
        \draw[] (5.5,-0.5) to (-0.5,-0.5);
        \draw[] (-0.5,-0.5) to (-0.5,0.5);

        \end{scope}
        
        \begin{scope}[xshift=17.5cm,yshift=-2cm]
        \node (a) at (0,0) [vertex] {1};
        \node (b) at (1,0) [vertex] {2};
        \node (c) at (2,0) [vertex] {3};
        \node (d) at (3,0) [vertex] {1};
        \node (e) at (4,0) [change] {2};
        \node (f) at (5,0) [vertex] {3};
        
        \draw[] (-0.5,0.5) to (5.5,0.5);
        \draw[] (5.5,0.5) to (5.5,-0.5);
        \draw[] (5.5,-0.5) to (-0.5,-0.5);
        \draw[] (-0.5,-0.5) to (-0.5,0.5);
        \draw[] (c) to[bend right] (f);

        \end{scope}
        
        \begin{scope}[xshift=10.5cm,yshift=-4cm]
        \node (a) at (0,0) [vertex] {1};
        \node (b) at (1,0) [vertex] {2};
        \node (c) at (2,0) [vertex] {3};
        \node (d) at (3,0) [vertex] {1};
        \node (e) at (4,0) [vertex] {1};
        \node (f) at (5,0) [vertex] {3};
        
        \draw[] (-0.5,0.5) to (5.5,0.5);
        \draw[] (5.5,0.5) to (5.5,-0.5);
        \draw[] (5.5,-0.5) to (-0.5,-0.5);
        \draw[] (-0.5,-0.5) to (-0.5,0.5);
        \draw[] (c) to[bend right] (f);

        \end{scope}
        
        \begin{scope}[xshift=17.5cm,yshift=-4cm]
        \node (a) at (0,0) [vertex] {1};
        \node (b) at (1,0) [vertex] {2};
        \node (c) at (2,0) [vertex] {3};
        \node (d) at (3,0) [vertex] {2};
        \node (e) at (4,0) [vertex] {1};
        \node (f) at (5,0) [vertex] {3};
        
        \draw[] (-0.5,0.5) to (5.5,0.5);
        \draw[] (5.5,0.5) to (5.5,-0.5);
        \draw[] (5.5,-0.5) to (-0.5,-0.5);
        \draw[] (-0.5,-0.5) to (-0.5,0.5);
        \draw[] (c) to[bend right] (f);
        \draw[] (d) to[bend right] (e);
        \end{scope}
        
        \draw[->,thick] (16.5,-.5) to (17,-2);
        \draw[->,thick,dashed] (16.5,-0.5) to (16,-2);
        \draw[->,thick,dashed] (20,-2.5) to (16,-3.5);
        \draw[->,thick,dashed] (20,-2.5) to (20,-3.5);
        
        \end{tikzpicture}
        \end{center}
        \caption{The proof of Lemma~\ref{lem:swaps}}
        \label{fig1}
    \end{figure}

We may now successively process each vertex of $I$: for each $v\in I$, we apply
Lemma~\ref{lem:swaps} to make $v$ $\alpha$-good, then we identify vertices in
$N(v)$ so that $v$ has degree at most $4$ in the resulting graph $N'$, and
moreover, $N'-v$ is planar. Note that since we avoid the color of $v$ in a
K-change by Lemma~\ref{lem:swaps}, the vertices of $I$ are never recolored at
any stage of the construction. We now formalize this intuition.

Let $N$ be a plane graph. Let $\varphi$ be a $5$-coloring of $N$. For a vertex
$v$ of $N$ of degree at most $6$ such that $v$ is $\varphi$-good if $d(v) = 6$,
we say that $(N', \varphi')$ is the result of \emph{collapsing} $(N, v,
\varphi)$ if
\begin{itemize}
  \itemsep0em
\item in the case $1 \leq d(v) \leq 4$, $N' = N$ and $\varphi' = \varphi$;
\item in the case $d(v) = 5$, $N'$ is the graph obtained from $N$ by identifying
  two neighbors $u$ and $w$ of $v$ with $\varphi(u) = \varphi(w)$ into a new
  vertex $z$ and $\varphi'$ is the coloring obtained from $\varphi$ by setting
  $\varphi'(z)=\varphi(u)=\varphi(w))$;
\item in the case $d(v) = 6$, $N'$ is the graph obtained from $N$ by either
  \begin{itemize}
    \itemsep0em
  \item identifying three neighbors $u$, $w$, $z$ of $v$ with $\varphi(u) =
    \varphi(w) = \varphi(z)$ into a new vertex $x$. In this case, $\varphi'$ is
    obtained from $\varphi$ by setting $\varphi'(x) = \varphi(u)$, or
  \item for two non-overlapping pairs $(u, w)$ and $(x, z)$ of neighbors of $v$
    with $\varphi(u) = \varphi(w)$ and $\varphi(x) = \varphi (z)$, identifying
    $u$, $w$ into a new vertex $s$ and $x$, $z$ into a new vertex $t$, and
    defining $\varphi'$ by setting $\varphi'(s) = \varphi(u)$ and $\varphi'(t) =
    \varphi(x)$.
  \end{itemize} 
\end{itemize}

Since $v$ is $\varphi$-good, it should be immediate that $N'$ and $\varphi'$ are
well-defined, and that $N' - \{v\}$ is planar. We now focus on showing that we
can extend a given recoloring sequence of the collapsed graph to a recoloring
sequence of the original graph.

\begin{lemma}\label{lem:extend}
  Let $N$ be a plane graph with a $5$-coloring $\varphi$. Let $v$ be a vertex of
  $N$ of degree at most $6$ such that $v$ is $\varphi$-good if $d(v) = 6$. Let
  $(N', \varphi')$ be the result of collapsing $(N, v, \varphi)$. Then every
  sequence of K-changes in $N'- \{v\}$ starting from $\varphi' \restriction
  (V(N') \setminus \{v\})$ extends to a sequence of K-changes in $N$ starting
  from $\varphi$. Moreover, each vertex of $N-v$ changes its color as many times
  as in $N'-v$, and $v$ changes its color at most once every time one of its
  neighbors in $N'$ changes its color.
\end{lemma}

\begin{proof}
  Each time a neighbor $w$ of $v$ is recolored in $N'-v$, we may use the same
  K-change in $N'$ unless it involves the color of $v$ and there is another
  neighbor $u$ of $v$ of the same color as $w$.  In this case, since at most 3
  colors appear on $N_{N'}(v)$ (as $v$ has degree 4 in $N'$) we can precede this
  K-change by first recoloring $v$ to a color not appearing in its
  neighborhood. This shows that any sequence of K-changes in $N' - \{v\}$
  extends to a sequence in $N'$.  To extend the sequence to $N$, observe that we
  can simulate in $N$ a K-change in $N'$ at a vertex $w$ by performing a
  K-change at each vertex that was identified to form $w$.  Clearly, each vertex
  of $N-v$ changes its color as many times as in $N'-v$, and $v$ changes its
  color at most once every time one of its neighbors in $N'$ changes its color,
  which concludes.
\end{proof}

We now successively apply Lemma~\ref{lem:extend} to each vertex of $I$ in $G$,
so that every vertex of $I$ becomes $\alpha$-good.
\begin{lemma}\label{lem:main}
  Let $N$ be a plane graph with a $5$-coloring $\varphi$. If $I$ is an
  $\varphi$-monochromatic independent set of vertices of degree at most $6$ in
  $N$, then there is a $5$-coloring $\psi$ of $N$ for which $I$ is $\psi$-good
  and a sequence of K-changes from $\varphi$ to $\psi$ that changes the color of
  each vertex at most $3|I|$ times.
\end{lemma}

\begin{proof}
  We shall prove by induction on $|V(N)|$ the stronger claim that there is such
  a sequence that avoids the color of $I$.

  By Lemma~\ref{lem:swaps}, we can assume, up to at most three K-changes, that
  $I$ contains a vertex $v$ of degree at most $5$ or a vertex of degree $6$ that
  is $\varphi$-good. So we can let $(N', \varphi')$ be the result of collapsing
  $(N, v, \varphi)$.  Since $N'' = N' - \{v\}$ is planar, we can apply our
  induction hypothesis to $N''$ (with $I - \{v\}$ instead of $I$ and $\varphi'
  \restriction N''$ instead of $\varphi$) to find a sequence of K-changes in
  $N''$ that avoids the color of $v$ and that transforms $\varphi' \restriction
  N''$ to some $5$-coloring $\varphi''$ of $N''$ so that the following holds:
  \begin{itemize}
  \item $I - \{v\}$ is $\varphi''$-good, and
  \item each vertex changes its color at most $3(|I| - 1)$ times.
  \end{itemize}  
  By Lemma~\ref{lem:extend}, this sequence extends to a sequence in
  $G$. Moreover, $v$ does not change its color (as the sequences avoids
  $\alpha(v)$), and every other vertex changes its color at most $3(|I|-1)+3=
  3|I|$ times. This completes the proof.
\end{proof}

\subsubsection*{Step 3: Induction.}

By Lemma \ref{lem:main}, we can assume that $I$ is $\alpha$-good. Let $n =
|V(G)|$ and $f(n)$ be the maximum number of times a vertex in $G$ is involved in
a K-change. Write $I = \{v_1, \dots, v_m\}$, set $G_1 = G$, $\psi_1 = \alpha$
and, for $i = 2, \dots, m + 1$, let $H_{i-1} = G_i - \{v_1, \dots, v_{i-1}\}$
where $(G_i, \psi_i)$ is the result of collapsing $(G_{i-1}, v_{i-1},
\psi_{i-1})$. Then the final graph $H_{m}$ is planar. Hence, by the induction
hypothesis combined with Lemma \ref{lem:I} (with $H_m$ in place of $G$), there
is a sequence of K-changes from $\psi_{m+1} \restriction V(H_{m})$ to some
$4$-coloring $\gamma'$ of $H_{m}$ on colors $\{1, \dots, 5\} \setminus
\alpha(I)$ where each vertex changes its color at most $f(n-|I|)$ times. By
successively applying Lemma~\ref{lem:extend} to $H_{m}$ etc. up until $H_1$ this
same sequence extends to a sequence in $G$ from $\alpha$ to some $5$-coloring
$\psi$ of $G$ where $\psi \restriction (G - I)$ uses only colors $\{1, \dots,
4\}$. Now, recalling Step 4 verbatim, by definition, $I \subset B$ for some
color class $B$ of $\beta$; so we can recolor each vertex of $B$ to color
$5$. Finally, noting that $G - B$ is a 3-colorable planar graph, we then apply
Proposition~\ref{prop:m1} to recolor the remaining vertices in $G - B$ to their
color in $\beta$.

\subsubsection*{Complexity analysis.}

By Lemma~\ref{lem:main}, the color of each vertex is changed at most $3|I|$
times in order to reach a 5-coloring $\psi$ in which $I$ is $\psi$-good. During
the induction step, by Lemma~\ref{lem:extend}, the color of each vertex in $I$
is changed at most $4f(n - |I|)$ times, while the color of the other vertices
changes at most $f(n-|I|)$ times. By Proposition~\ref{prop:m1}, the final step
requires at most $O(n^2)$ changes of color per vertex. We deduce that $f(n)$
satisfies the recurrence
\[f(n)\leqslant 3|I| + 4f(n-|I|) +O(n^2) \leq 4f\left(n-\frac{n}{140}\right)+O(n^2).\]
The master theorem then yields that each vertex changes its color
$O(n^{\log_{\frac{140}{139}}(4)})=O(n^{194})$ times, hence the sequence has
length at most $O(n^{195})$, which concludes the proof.

\section{Acknowledgements}

This work was supported by Agence Nationale de la Recherche (France) under
research grants ANR DIGRAPHS ANR-19-CE48-0013-01 and ANR prpdfoject GrR
(ANR-18-CE40-0032) and by APVV-19-0308 and VEGA 1/0743/21.

\printbibliography
\end{document}